\documentclass[a4paper,12pt]{article}
\usepackage{amsmath,amssymb,amsthm,mathtools,geometry,hyperref,bm,cite,graphicx}
\hypersetup{colorlinks=true,linkcolor=blue,urlcolor=blue,hypertexnames=false}
\usepackage{algorithm}
\usepackage{algpseudocode}
\usepackage{placeins}
\theoremstyle{plain}
\newtheorem{theorem}{Theorem}[section]
\newtheorem{lemma}{Lemma}[section]
\newtheorem{proposition}{Proposition}[section]

\theoremstyle{definition}
\newtheorem{definition}{Definition}[section]

\theoremstyle{remark}
\newtheorem{remark}{Remark}[section]

\newcommand{\R}{\mathbb R}
\newcommand{\eps}{\varepsilon}
\newcommand{\norm}[1]{\left\lVert #1 \right\rVert}
\newcommand{\Ker}{\operatorname{Ker}}
\newcommand{\Ran}{\operatorname{Ran}}
\newcommand{\diag}{\operatorname{diag}}
\newcommand{\Span}{\operatorname{span}}

\newcommand{\bx}{\bm{x}}
\newcommand{\bv}{\bm{v}}
\newcommand{\bq}{\bm{q}}
\newcommand{\bp}{\bm{p}}
\newcommand{\br}{\bm{r}}
\newcommand{\bu}{\bm{u}}
\newcommand{\by}{\bm{y}}
\newcommand{\bz}{\bm{z}}
\newcommand{\bY}{\bm{Y}}
\newcommand{\bZ}{\bm{Z}}
\newcommand{\boldeta}{\bm{\eta}}

\newcommand{\bw}{\bm{w}}

\newcommand{\bB}{\bm{B}}
\newcommand{\bE}{\bm{E}}
\newcommand{\bF}{\bm{F}}
\newcommand{\bG}{\bm{G}}
\newcommand{\bH}{\bm{H}}
\newcommand{\bN}{\bm{N}}
\newcommand{\bR}{\bm{R}}
\newcommand{\bxi}{\bm{\xi}}
\newcommand{\bchi}{\bm{\chi}}
\newcommand{\bdelta}{\bm{\delta}}
\newcommand{\bzero}{\bm{0}}
\newcommand{\be}{\bm{e}}
\newcommand{\balpha}{\bm{\alpha}}
\newcommand{\bbeta}{\bm{\beta}}

\newcommand{\bPsi}{\bm{\Psi}}
\newcommand{\bXi}{\bm{\Xi}}
\newcommand{\cI}{\mathcal I}
\newcommand{\mE}{\mathcal E}

\title{A Curve-Reference Exponential Integrator for Three-Dimensional Charged-Particle Dynamics under a Strong Nonconstant Magnetic Field}
\author{Zhirui Shen and Bin Wang}
\date{}

\begin{document}
\maketitle

\begin{abstract}
We study a three-dimensional charged-particle dynamics model in the nonrelativistic momentum formulation and construct a curve-reference exponential integrator (CREI) for this system. Since the zero eigenvalue of linear section is not semisimple,  the system after usual change of variables does not meet the spectral condition used in the classical locally linearized extended exponential integrator (LLEEI). A kernel-range decomposition of the dominant magnetic matrix and a further linear transformation reduce the equation to a system with a semisimple zero block and one constant skew-symmetric oscillatory block.

The CREI expands the nonlinear term along the exact oscillatory curve rather than at a fixed point. Retaining monomials through degree \(k\) defines CREI\((k+1)\). We prove uniform local and global error bounds and present reproducible numerical experiments for convergence in \(h\) and uniformity in \(\eps\).
\end{abstract}

\section{Introduction}

The three-dimensional charged-particle dynamics (CPD) in a strong nonconstant magnetic field is a standard model of highly oscillatory ordinary differential equations. In the nonrelativistic case, a representative model in the highly oscillatory regime is
\begin{align*}
    {\bx}'(t)&=\bv(t),\\
    {\bv}'(t)&=\bv(t)\times\bB_{\eps}(\bx)+\bE(\bx),\quad 0<\eps\le1,
\end{align*}
where \(\bx,\bv\in\R^3\), \(\bE\) is a bounded electric field, and \(\bB_{\eps}\) is a strong spatially varying magnetic field. We separate \(\bB_{\eps}\) into a dominant constant part \(\eps^{-1}\bB_0\neq\mathbf{0}\) and a bounded spatial perturbation \(\bB_1(\bx)\), and obtain
\begin{equation}\label{eq:cpd}
    \begin{aligned}
        {\bx}'(t)&=\bv(t),\\
        {\bv}'(t)&=\bv(t)\times(\eps^{-1}\bB_0)+\bv(t)\times\bB_1(\bx)+\bE(\bx).
    \end{aligned}
\end{equation}
This form also covers the maximal-ordering field \(\bB(\eps\bx)/\eps\). For a fixed point \(\bx_0\), write
\[
\frac1\eps\bB(\eps\bx)
=\frac1\eps\bB(\eps\bx_0)+\bB_{1,\eps}(\bx).
\]
Although \(\bB_{1,\eps}\) depends on \(\eps\), it is uniformly bounded on the region of interest when \(\bB\) is smooth. The proof uses only the corresponding uniform bounds, so the error constants remain independent of \(\eps\).

For simplicity the charge-to-mass factor has been absorbed into the fields.
The parameter \(\eps\) measures the inverse strength of the principal magnetic field. When \(\eps\) is small, the Lorentz force generated by \(\eps^{-1}\bB_0\) produces rapid rotations on the time scale \(O(\eps)\), whereas the electric force and the lower-order magnetic perturbation act on an \(O(1)\) time scale. This simultaneous presence of a short gyromotion scale and a long transport scale is one of the basic numerical difficulties in plasma physics, particle methods, and geometric simulation. The same computational issue appears in many other highly oscillatory models, including second-order oscillatory equations, molecular dynamics, wave-type equations in singular limits, and oscillatory Hamiltonian systems\cite{cohenJahnkeLorenzLubich2006,hairerLubichWanner2013,wuYouWang2013}.

Classical one-step methods are usually not appropriate in this regime unless the time step resolves the fast period. For explicit Runge--Kutta, Runge--Kutta--Nystrom, and related polynomial approximation methods, the local truncation analysis involves time derivatives of the exact solution, and these derivatives typically grow like negative powers of \(\eps\). This leads to error constants that deteriorate as \(\eps\to0\). Even when a method is stable, the accuracy may fail to be uniform in the oscillatory parameter. This observation motivated a long line of methods that treat oscillatory linear components more carefully. Gautschi's trigonometric-polynomial idea \cite{gautschi1961} and later Gautschi-type and trigonometric integrators \cite{hochbruckLubich1999,grimm2005,grimmHochbruck2006} were designed to incorporate exact or filtered oscillatory information into the numerical step. Modulated Fourier expansions provide a complementary analytical framework for long-time behavior and conservation properties of highly oscillatory systems \cite{hairerLubich2000,cohenHairerLubich2003,cohenHairerLubich2005,cohen2006}. For charged particles in strong magnetic fields, filtered Boris-type algorithms and geometric two-scale integrators have shown that carefully designed schemes can capture the physically relevant slow behavior without resolving every fast turn \cite{hairerLubichWang2020,wangZhao2023}.

The past two years have seen a continued effort to design exponential integrators whose error bounds are uniform with respect to the oscillatory parameter. Exponential integrators are natural candidates because the variation-of-constants formula allows the stiff linear part to be integrated exactly; the general theory and implementation of these methods are treated in \cite{coxMatthews2002,krogstad2005,hochbruckOstermannSchweitzer2009,caliariOstermann2009,hochbruckOstermann2010}. For oscillatory second-order systems, exponential Runge--Kutta--Nystrom methods, trigonometric collocation methods, and Filon-type constructions exploit the same principle in different forms \cite{wuYouShiWang2010,wangWuXia2013,wangLiuWu2013}. Recent work has also refined uniform bounds for exponential and parareal-type methods in oscillatory regimes \cite{wangJiang2024}. At the same time, uniformly accurate multiscale time integrators (EIs), two-scale formulations, stroboscopic averaging, heterogeneous multiscale methods, and nested Picard-type schemes have supplied powerful alternatives for problems with clear scale separation \cite{baoDongZhao2014,zhao2017,caiGuo2021,e2003,abdulleEEngquistVandenEijnden2012,calvoSanzSerna2010,chartierMakazagaMuruaVilmart2014,chartierLemouMehatsVilmart2020,chartierCrouseillesZhao2018,crouseillesLemouMehatsZhao2017}. These methods differ substantially in construction, but they share a common goal: the step size and the error estimate should not degrade merely because the solution contains a fast but structurally simple oscillation.

The locally linearized extended exponential integrator (LLEEI) was introduced by Qi, Deng, and Zhu in \cite{qiDengZhu2025} for highly oscillatory equations with a dominant linear part. At each step, classical LLEEI fixes the current numerical value as a reference point, Taylor expands the nonlinear term at that point, augments the state by the retained monomials, and advances the resulting finite linear system by an exponential. This construction keeps the dominant oscillation inside the exponential and obtains higher order by increasing the polynomial degree. Related local-linear-extension ideas were subsequently applied to charged-particle dynamics in \cite{wangWangZhu2026}. In particular, \cite{qiDeng2026} studies the standard two-dimensional CPD with a uniform magnetic field perpendicular to the plane. In contrast, \eqref{eq:cpd} is three-dimensional and includes both motion parallel to the dominant field and the nonlinear magnetic term \(\bv\times\bB_1(\bx)\).

The fixed-point reference is restrictive when \(h/\eps\) is not small. Even for the homogeneous equation, the displacement \((e^{(t-t_n)A/\eps}-I)\bu_n\) need not be \(O(h)\); consequently, a Taylor expansion at \(\bu_n\) does not yield a uniform remainder estimate. The present paper replaces the fixed point by the homogeneous curve \(e^{(t-t_n)A/\eps}\bu_n\). The resulting method is called the curve-reference exponential integrator (CREI). The monomial extension and projection are inherited from LLEEI, while the moving reference removes the known rapid oscillation before the Taylor expansion is formed.

For the three-dimensional system \eqref{eq:cpd}, a second obstruction appears before CREI can be applied. In the usual position-momentum variables, the zero eigenvalue of the dominant matrix is not semisimple, so the exponential of its monomial extension can grow polynomially. Let \(\mathcal B\bw=\bw\times\bB_0\). We use the decomposition \(\R^3=\Ker\mathcal B\oplus\Ran\mathcal B\) and the charged-particle transformation
\[
\bq=\bx,
\qquad
\bp=2\eps\bv-\mathcal B\bq,
\]
which is the form suggested by the Hamiltonian-like representation of magnetic rotation. After a series of technical changes of variables, we can transform the equation with a highly oscillatory factor into a simple form
\[
\dot{\by}=\bF(\by,\boldeta),
\qquad
\dot{\boldeta}=\eps^{-1}\Omega\boldeta+\bG(\by,\boldeta).
\]
The zero block of the transformed dominant matrix is semisimple, and its only nonzero block is the constant skew-symmetric matrix \(\Omega\). This is the algebraic structure required for uniformly bounded extended exponentials.

In the nonrelativistic model, the frequency generated by \(\Omega/\eps\) is constant. Hence the CREI reference is the exact homogeneous rotation \(e^{(t-t_n)\Omega/\eps}\boldeta_n\), without phase prediction. We prove the smoothness of the transformed nonlinear term, the boundedness of the extended oscillatory exponential, and global \(O(h^k)\) error bounds. Section \ref{sec:lleei} recalls LLEEI, defines CREI, and derives the transformed CPD system. Sections \ref{sec:small} and \ref{sec:large} establish the error estimates. Section \ref{sec:numerics} reports the numerical experiments.

\section{From LLEEI to CREI and Preparatory Transformations}\label{sec:lleei}

We first fix the ordering and projection used for monomial extensions. We then recall classical LLEEI and define CREI before transforming the charged-particle equation into the form used in the analysis.

\subsection{Multi-indices and ordering}

In the following, we assume \(\bu=(u_1,\ldots,u_d)^\top\in\R^d\) and \(\balpha\in\mathbb N^d\). we first define the multi-indices and ordering.

\begin{definition}\label{def:ordering}
For \(k\ge0\), let
\begin{itemize}
    \item Multi-indices:\[
    |\balpha|=\alpha_1+\cdots+\alpha_d,\qquad\cI_d^k=\{\balpha\in\mathbb N^d:\ |\balpha|\le k\}.\]
    Conventionally, for $k=0$, we define the \(\cI_d^0\) as the singleton set containing only the empty index with no components.
    \item Standard ordering: 
    We order \(\mathcal I_d^k\) by the graded lexicographic order.
    For \(\balpha,\bbeta\in\mathcal I_d^k\), we write
    \[
    \balpha\prec\bbeta
    \]
    if either
    \[
    |\balpha|<|\bbeta|,
    \]
    or
    \[
    |\balpha|=|\bbeta|
    \quad\text{and}\quad
    \exists\, r\in\{1,\ldots,d\}\ \text{s.t.}\ 
    \alpha_1=\beta_1,\ldots,\alpha_{r-1}=\beta_{r-1},
    \quad
    \alpha_r<\beta_r .
    \]
    Thus the indices are first sorted by total degree and then lexicographically
    within each fixed degree. For example, when \(d=2\) and \(k=3\), this ordering gives
    \[
    \begin{aligned}
    \mathcal I_2^3
    =
    \{&
    (0,0),\\
    &(0,1),(1,0),\\
    &(0,2),(1,1),(2,0),\\
    &(0,3),(1,2),(2,1),(3,0)
    \}.
    \end{aligned}
    \]
    \item Unit multi-indices:\[
     (\be_j)_\ell=
     \begin{cases}
     1, & \ell=j,\\
     0, & \ell\ne j,
     \end{cases}
     \qquad \ell=1,\ldots,d.
     \]
    \item Standard exponent-vector multi-index notation:
    \[
    \balpha!=\alpha_1!\cdots\alpha_d!,\qquad
    \bu^{\balpha}=\prod_{j=1}^d u_j^{\alpha_j}.
    \]
    \item Projection operator: Suppose $i\le j$,
     \[
     \Pi_i^j\bm u=(u_i,u_{i+1}\ldots,u_j)^\top.
     \]
    For the extended vectors used in this paper, the ordered basis starts with the constant monomial, so the constant component is \(\Pi_1^1\). In the six-dimensional transformed CPD variables, the degree-one monomial block is \(\Pi_2^7\), etc..
    \item Partial differential operator:
    \[
    \frac{\partial^\alpha}{\partial \bx^{\balpha}}:=
     \begin{cases}
     \mathrm{id}, & \text{if } |\balpha| = 0, \\
     \frac{\partial^j}{\partial x_{\alpha_1} \cdots \partial x_{\alpha_j}}, & \text{if } |\balpha| = j \geq 1.
     \end{cases}
    \]
\end{itemize}
\end{definition}

For the exponent-vector notation used below, we write
\[
\partial^{\balpha}f(\bu)
:=
\frac{\partial^{|\balpha|}f(\bu)}
{\partial u_1^{\alpha_1}\cdots\partial u_d^{\alpha_d}},
\qquad
\partial^{\bzero}f=f.
\]

\subsection{Classical LLEEI}\label{sec:classical-lleei}

The locally linearized extended exponential integrator (LLEEI) was introduced
by Qi, Deng, and Zhu in \cite{qiDengZhu2025} for equations of the form
\begin{equation}\label{eq:general-lleei}
\dot{\bu}=\frac1\eps A\bu+\bF(\bu),\qquad \bu(t)\in\R^d.
\end{equation}
Its purpose is to keep the action of the dominant matrix \(A/\eps\) inside an
exponential while representing the local nonlinear contribution by a finite
linear system. Fix the step \([t_n,t_{n+1}]\), where \(n\) is the time-step
index, and let \(\bu_n\) be the numerical value at \(t_n\). A classical LLEEI
step consists of four operations:
\begin{enumerate}
    \item use \(\bu_n\) as a fixed reference and set \(\bxi_n=\bu-\bu_n\);
    \item replace \(\bF(\bu_n+\bxi_n)\) by its Taylor polynomial of degree \(r\);
    \item introduce all monomials \(\bxi_n^{\bbeta}\), \(\bbeta\in\cI_d^r\), and
    differentiate them to obtain a finite linear system;
    \item advance that system and recover \(\bxi_n\) by \(\Pi_2^{d+1}\).
\end{enumerate}
Indeed,
\[
\bF(\bu_n+\bxi_n)=
\sum_{|\balpha|\le r}
\frac{\partial^{\balpha}\bF(\bu_n)}{\balpha!}\bxi_n^{\balpha}
+\bR_n^{[r+1]}(\bxi_n).
\]

The reason for collecting monomials is that the truncated right-hand side is
polynomial in \(\bxi_n\), but linear in the list of its monomials. After
differentiating this list and discarding degrees larger than \(r\), one obtains
a closed finite linear system. Its exponential can be used directly, and the
original deviation is recovered by projection. This is the central extension
mechanism of LLEEI \cite{qiDengZhu2025}.

For example, let \(d=1\), \(r=2\), and suppose the truncated deviation equation
is
\[
\dot\xi_n=g_{0,n}+g_{1,n}\xi_n+g_{2,n}\xi_n^2,
\qquad \xi_n(t_n)=0.
\]
If \(A=a\in\R\), then
\[
g_{0,n}=\frac{a}{\eps}u_n+F(u_n),\qquad
g_{1,n}=\frac{a}{\eps}+F'(u_n),\qquad
g_{2,n}=\frac12F''(u_n).
\]
With \(\bXi_n^{[2]}=(1,\xi_n,\xi_n^2)^\top\), differentiation gives
\[
\frac{d}{dt}\bXi_n^{[2]}
=
\begin{pmatrix}
0&0&0\\
g_{0,n}&g_{1,n}&g_{2,n}\\
0&2g_{0,n}&2g_{1,n}
\end{pmatrix}
\bXi_n^{[2]}.
\]
The omitted term is \(2g_{2,n}\xi_n^3\), whose degree exceeds \(r=2\).
Thus a quadratic scalar equation has become a three-dimensional linear system,
and \(\Pi_2^2\bXi_n^{[2]}=\xi_n\).

In general, after the Taylor remainder is omitted, the ordered monomial vector
\[
\bXi_n^{[r]}(t)=
\bigl(\bxi_n(t)^{\bbeta}\bigr)_{\bbeta\in\cI_d^r},\qquad
\bXi_n^{[r]}(t_n)=(1,0,\ldots,0)^\top,
\]
satisfies a finite linear equation
\[
\dot{\bXi}_n^{[r]}=M_n^{[r]}\bXi_n^{[r]},
\qquad
\Pi_2^{d+1}\bXi_n^{[r]}=\bxi_n.
\]
Here \(M_n^{[r]}\) is the coefficient matrix generated on the \(n\)-th step
by differentiating the retained monomials. The index \(n\) records the
dependence on the reference value \(\bu_n\), while \([r]\) records the highest
retained monomial degree

The fixed reference is the limitation relevant here. For the homogeneous
solution \(e^{(t-t_n)A/\eps}\bu_n\), the displacement from \(\bu_n\) is
\((e^{(t-t_n)A/\eps}-I)\bu_n\), which need not be \(O(h)\) when
\(h/\eps\) is not small. Thus the Taylor remainder at \(\bu_n\) is not
uniformly controlled by a power of \(h\). Moreover, the dominant matrix of the
untransformed three-dimensional CPD system has a non-semisimple zero block, so
its extended exponential does not satisfy the boundedness condition used in
the LLEEI analysis.

\subsection{Curve-reference exponential integrator}\label{sec:crei}

CREI changes only the reference used before the Taylor expansion. For
\eqref{eq:general-lleei}, define the homogeneous curve and the deviation by
\[
\widehat{\bu}_n(t)=e^{(t-t_n)A/\eps}\bu_n,
\qquad
\bdelta_n(t)=\bu(t)-\widehat{\bu}_n(t).
\]
Then
\begin{equation}\label{eq:crei-delta-exact}
\dot{\bdelta}_n=\frac1\eps A\bdelta_n+
\bF\bigl(\widehat{\bu}_n(t)+\bdelta_n\bigr),
\qquad \bdelta_n(t_n)=\bzero.
\end{equation}
The homogeneous curve contains the known motion generated by \(A/\eps\).
Consequently, the Taylor variable \(\bdelta_n\) measures only the correction
produced by \(\bF\), rather than the full displacement from the fixed point
\(\bu_n\). This is the step that makes the remainder estimate uniform when the
homogeneous solution travels far from \(\bu_n\).

For an integer \(r\ge1\), Taylor expansion along the curve gives
\begin{equation}\label{eq:crei-taylor}
\bF\bigl(\widehat{\bu}_n(t)+\bdelta\bigr)=
\sum_{|\balpha|\le r}
\bm{c}_{\balpha}^{(n)}(t)\bdelta^{\balpha}
+\bR_n^{[r+1]}(t,\bdelta),
\qquad
\bm{c}_{\balpha}^{(n)}(t)=
\frac{\partial^{\balpha}\bF(\widehat{\bu}_n(t))}{\balpha!}.
\end{equation}
Unlike classical LLEEI, these Taylor coefficients depend on \(t\), because the
reference position moves along \(\widehat{\bu}_n(t)\). The monomial extension
itself is unchanged. In the scalar quadratic example above, one only replaces
\(g_{j,n}\) by \(g_{j,n}(t)\); the same three monomials
\((1,\delta_n,\delta_n^2)\) then satisfy a time-dependent linear system.

After omitting the remainder, set
\[
\bZ_n^{[r]}(t)=
\bigl(\bdelta_n(t)^{\bbeta}\bigr)_{\bbeta\in\cI_d^r},
\qquad
\bZ_n^{[r]}(t_n)=(1,0,\ldots,0)^\top.
\]
For \(1\le|\bbeta|\le r\), monomial differentiation yields
\begin{align}
\frac{d}{dt}\bdelta_n^{\bbeta}
&=\frac1\eps\sum_{j,\ell=1}^d
\beta_jA_{j\ell}
\bdelta_n^{\bbeta-\be_j+\be_\ell}\nonumber\\
&\quad+
\sum_{j=1}^d\beta_j
\sum_{\substack{|\balpha|\le r\\|\bbeta|-1+|\balpha|\le r}}
c_{j,\balpha}^{(n)}(t)
\bdelta_n^{\bbeta-\be_j+\balpha},
\label{eq:crei-extended-components}
\end{align}
where terms with a negative multi-index component are omitted. Hence
\[
\dot{\bZ}_n^{[r]}(t)=\widehat M_n^{[r]}(t)\bZ_n^{[r]}(t),
\qquad
\Pi_2^{d+1}\bZ_n^{[r]}(t)=\bdelta_n(t).
\]
Here \(\widehat M_n^{[r]}(t)\) is determined by
\eqref{eq:crei-extended-components}. Again, \(n\) is the time-step index and
\([r]\) is the retained degree; the hat distinguishes the moving-reference
matrix from the fixed-reference matrix \(M_n^{[r]}\) in Section
\ref{sec:classical-lleei}. Let \(\widehat\Phi_n^{[r]}(t,s)\) be its
state-transition matrix, defined by
\[
\partial_t\widehat\Phi_n^{[r]}(t,s)
=\widehat M_n^{[r]}(t)\widehat\Phi_n^{[r]}(t,s),
\qquad
\widehat\Phi_n^{[r]}(s,s)=I.
\]
One CREI step is therefore
\begin{equation}\label{eq:crei-step}
\bu_{n+1}=e^{hA/\eps}\bu_n+
\Pi_2^{d+1}\widehat\Phi_n^{[r]}(t_{n+1},t_n)
(1,0,\ldots,0)^\top.
\end{equation}

The method retaining all monomials of degree at most \(r\) is denoted by
CREI\((r+1)\). In particular, retaining the Taylor expansion through degree
\(k\) defines CREI\((k+1)\); CREI2, CREI3, and CREI4 retain degrees
\(1,2,3\), respectively. In the error analysis below, CREI\((k)\) therefore
uses degree \(k-1\). The difference from classical LLEEI is precisely the
replacement of the fixed point \(\bu_n\) by the homogeneous curve
\(\widehat{\bu}_n(t)\).

\subsection{Change of variables}\label{sec:qp}
Note that the CPD \eqref{eq:cpd} is not in canonical form. So in order to solve it in our context, we introduce a change of variable
\begin{equation}\label{eq:qp-transform}
\bq=\bx,\qquad \bp=2\eps\bv-\mathcal B\bq,
\end{equation}
and
\begin{equation}\label{eq:v-inverse}
\bv=\frac{\bp+\mathcal B\bq}{2\eps}.
\end{equation}
Define
\[
\bH(\bq,\bp)=\bE(\bq)+\frac{\bp+\mathcal B\bq}{2\eps}\times\bB_1(\bq).
\]

\begin{proposition}\label{prop:qp-system}
Under the change of variables \eqref{eq:qp-transform} and \eqref{eq:v-inverse}, the system \eqref{eq:cpd} is equivalent to
\begin{equation}\label{eq:qp-matrix}
\frac{d}{dt}\binom{\bq}{\bp}
=\frac1{2\eps}
\begin{pmatrix}
\mathcal B&I\\
\mathcal B^2&\mathcal B
\end{pmatrix}
\binom{\bq}{\bp}
+\binom{\bzero}{2\eps\bH(\bq,\bp)}.
\end{equation}
Equivalently, in components,
\begin{equation}\label{eq:qp-system}
\dot{\bq}=\frac1{2\eps}(\mathcal B\bq+\bp),\qquad
\dot{\bp}=\frac1{2\eps}(\mathcal B^2\bq+\mathcal B\bp)+2\eps\bH(\bq,\bp).
\end{equation}
\end{proposition}

\begin{proof}
Differentiate \(\bq=\bx\) to get \(\dot{\bq}=\bv\), and use \eqref{eq:v-inverse}. For \(\bp\), we have
\[
\dot{\bp}=2\eps\dot{\bv}-\mathcal B\dot{\bq}.
\]
From \eqref{eq:cpd},
\[
2\eps\dot{\bv}=2\mathcal B\bv+2\eps\bH(\bq,\bp).
\]
Substituting \(\dot{\bq}=\bv\) gives
\[
\dot{\bp}=2\mathcal B\bv+2\eps\bH(\bq,\bp)-\mathcal B\bv
=\mathcal B\bv+2\eps\bH(\bq,\bp).
\]
Then using \(\bv=(\bp+\mathcal B\bq)/(2\eps)\) yields \eqref{eq:qp-system}.
\end{proof}

For a  fixed \(\bH\), the dominant matrix in \eqref{eq:qp-system} is proportional to
\[
A_{\rm na}:=\begin{pmatrix}
\mathcal B&I\\
\mathcal B^2&\mathcal B
\end{pmatrix},
\]
This matrix still fails the spectral condition used below. Indeed, let
\(\be_0\in\Ker\mathcal B\) be a unit vector. Since
\(\mathcal B\be_0=\bzero\), we have \(A_{\rm na}(\be_0,\bzero)^\top=\bzero\) and \(A_{\rm na}(\bzero,\be_0)^\top=(\be_0,\bzero)^\top\).
Thus \((\bzero,\be_0)^\top\) is a generalized eigenvector associated with the zero eigenvalue, while \((\be_0,\bzero)^\top\) is an eigenvector. Hence the zero eigenvalue has a Jordan chain of length at least two, so the zero block of \(A_{\rm na}\) is not semisimple. This nilpotent part would produce polynomial growth in the exponential of the leading matrix and also in its monomial extensions.

To solve the above problems, we need to analyze the properties of the kernel and image of $\mathcal B$ and perform another change of variables based on this. Since \(\mathcal B\bz=\bz\times\bB_0\), one has
\[
\Ker\mathcal B=\Span\{\be_0\},
\qquad
\Ran\mathcal B=\be_0^\perp .
\]
Choose a unit vector \(\be_1\in\be_0^\perp\). Let \(\be_2=\be_1\times\be_0\), \(Q_\perp=(\be_1,\be_2)\), \(Q=(\be_0,Q_\perp)\) and denote \(\omega=\norm{\bB_0}\). Then \(Q\) is orthogonal and
\[
Q^\top\mathcal BQ
=
\begin{pmatrix}
0&0\\
0&\Omega
\end{pmatrix},
\qquad
\Omega=Q_\perp^\top\mathcal BQ_\perp .
\]
With the above orientation of \(Q_\perp\), we have \(\mathcal B\be_1=\omega\be_2\) and \(\mathcal B\be_2=-\omega\be_1\).
Hence
\[
\Omega
=
\omega
\begin{pmatrix}
0&-1\\
1&0
\end{pmatrix}.
\]
Thus the one-dimensional block \(\Ker\mathcal B\) is the zero block, while the
two-dimensional block \(\Ran\mathcal B\) is represented by the nonsingular
skew-symmetric matrix \(\Omega\).

Introduce
\begin{equation}\label{eq:r-def}
\br=\bp+\mathcal B\bq=2\eps\bv .
\end{equation}
Then \eqref{eq:qp-system} becomes
\begin{equation}\label{eq:qr-system}
\dot{\bq}=\frac1{2\eps}\br,\qquad
\dot{\br}=\frac1\eps\mathcal B\br+2\eps\bH(\bq,\br),
\end{equation}
where
\[
\bH(\bq,\br)=\bE(\bq)+\frac{\br}{2\eps}\times\bB_1(\bq).
\]
We now separate the variables according to the decomposition
\(\R^3=\Ker\mathcal B\oplus\Ran\mathcal B\), and correct the perpendicular
position by the fast variable \(\br\). Define
\begin{equation}\label{eq:new-vars}
\sigma=\be_0^\top\bq,\qquad
\nu=\frac1{2\eps}\be_0^\top\br,\qquad
\bZ=\frac12\Omega^{-1}Q_\perp^\top\br,\qquad
\boldeta=\frac{\bZ}{\eps},\qquad
\bY=Q_\perp^\top\bq-\bZ.
\end{equation}
Collecting
\[
\bu=(\sigma,\nu,\bY^\top,\boldeta^\top)^\top,\qquad
\by=(\sigma,\nu,\bY^\top)^\top,
\]
we have a proposition about the inverse map of \((\bq,\br)\mapsto(\bu,\by)\).

\begin{proposition}\label{prop:inverse}
The transformation \eqref{eq:new-vars} is invertible and
\begin{equation}\label{eq:inverse}
\bq=\be_0\sigma+Q_\perp(\bY+\eps\boldeta),\qquad
\bp=2\eps\nu\,\be_0+Q_\perp\Omega(\eps\boldeta-\bY).
\end{equation}
Moreover,
\[
\bv=\nu\be_0+Q_\perp\Omega\boldeta.
\]
\end{proposition}

\begin{proof}
From \eqref{eq:new-vars}, we have \(Q_\perp^\top\bq=\bY+\bZ=\bY+\eps\boldeta\), and
\(\be_0^\top\bq=\sigma\). Hence
\[
\bq=\be_0\sigma+Q_\perp(\bY+\eps\boldeta).
\]
Also, \(Q_\perp^\top\br=2\Omega\bZ=2\eps\Omega\boldeta\), and
\(\be_0^\top\br=2\eps\nu\). Therefore
\[
\br=2\eps\nu\,\be_0+2\eps Q_\perp\Omega\boldeta,
\qquad
\bv=\frac{\br}{2\eps}
=
\nu\be_0+Q_\perp\Omega\boldeta.
\]
Finally, since \(\bp=\br-\mathcal B\bq\) and
\(\mathcal B\bq=Q_\perp\Omega(\bY+\eps\boldeta)\), we obtain
\[
\bp
=
2\eps\nu\,\be_0
+
2\eps Q_\perp\Omega\boldeta
-
Q_\perp\Omega(\bY+\eps\boldeta)
=
2\eps\nu\,\be_0+Q_\perp\Omega(\eps\boldeta-\bY).
\]
\end{proof}

\begin{remark}\label{rem:inverse-lip}
For \(0<\eps\le1\), the map
\[
(\sigma,\nu,\bY,\boldeta)\mapsto(\bq,\bp)
\]
defined in \eqref{eq:inverse} has an \(\eps\)-uniform Lipschitz constant on \(\R^6\), because the only \(\eps\)-dependent coefficients are \(\eps\) and \(2\eps\). Thus an \(O(h^k)\) error in \(\bu\) implies an \(O(h^k)\) error in \((\bq,\bp)\). If one reports the original velocity through \(\bv=(\bp+\mathcal B\bq)/(2\eps)\), the estimate becomes \(O(h^k/\eps)\) unless the transformed variables or the momenta are controlled in a stronger, \(\eps\)-weighted norm.
\end{remark}

\begin{theorem}[Transformed standard form]\label{thm:transformed}
The system \eqref{eq:cpd} is equivalent to
\begin{equation}\label{eq:transformed}
\begin{cases}
\dot\sigma=\nu,\\
\dot\nu=\be_0^\top\widehat{\bH}(\sigma,\nu,\bY,\boldeta),\\
\dot{\bY}=-\eps\Omega^{-1}Q_\perp^\top\widehat{\bH}(\sigma,\nu,\bY,\boldeta),\\
\dot{\boldeta}=\frac1\eps\Omega\boldeta+\Omega^{-1}Q_\perp^\top\widehat{\bH}(\sigma,\nu,\bY,\boldeta),
\end{cases}
\end{equation}
where
\[
\widehat{\bH}(\sigma,\nu,\bY,\boldeta)
=\bE(\bq)+\bv\times\bB_1(\bq),
\qquad
\bq=\be_0\sigma+Q_\perp(\bY+\eps\boldeta),\qquad
\bv=\nu\be_0+Q_\perp\Omega\boldeta.
\]
The zero eigenvalue of the leading matrix is semisimple, and the oscillatory block is the constant rotation \(\Omega\).
\end{theorem}

\begin{proof}
From \(\dot{\bq}=(2\eps)^{-1}\br\), since \(\be_0^\top\mathcal B=0\), we have
\[
\dot\sigma=\be_0^\top\dot{\bq}=\be_0^\top\frac1{2\eps}\br=\nu.
\]
For \(\nu\),
\[
\dot\nu=\frac1{2\eps}\be_0^\top\dot{\br}
=\frac1{2\eps^2}\be_0^\top\mathcal B\br+\be_0^\top\bH(\bq,\br)=\be_0^\top\bH(\bq,\br),
\]
Similarly,
\begin{align*}
\dot{\bZ}=&\frac12\Omega^{-1}Q_\perp^\top\dot{\br}\\
=&\frac1{2\eps}\Omega^{-1}Q_\perp^\top\mathcal B\br+\eps\Omega^{-1}Q_\perp^\top\bH\\
=&\frac1\eps\Omega\bZ+\eps\Omega^{-1}Q_\perp^\top\bH.
\end{align*}
Since \(\bZ=\eps\boldeta\) and \(\eps\) is fixed in time, this gives
\[
\dot{\boldeta}=\frac1\eps\Omega\boldeta+\Omega^{-1}Q_\perp^\top\bH.
\]
Finally,
\[
\dot{\bY}=Q_\perp^\top\dot{\bq}-\dot{\bZ}
=\frac1{2\eps} Q_\perp^\top\br-\dot{\bZ}
=-\eps\Omega^{-1}Q_\perp^\top\bH.
\]
This yields \eqref{eq:transformed}.
\end{proof}

Denote \(\by=(\sigma,\nu,\bY^\top)^\top\), we could write \eqref{eq:transformed} as
\begin{equation}\label{eq:y-z}
\dot{\by}=\bF(\by,\boldeta),\qquad
\dot{\boldeta}=\frac1\eps\Omega\boldeta+\bG(\by,\boldeta),
\end{equation}
Since \(\bq=\be_0\sigma+Q_\perp(\bY+\eps\boldeta)\) and \(\bv=\nu\be_0+Q_\perp\Omega\boldeta\), the functions \(\bF\) and \(\bG\) are smooth in \((\by,\boldeta)\) with \(\eps\)-uniform derivative bounds on compact sets.

Form \eqref{eq:y-z} ensures the singular term appears only in the \(\boldeta\)-equation, and its coefficient \(\Omega\) is a constant skew-symmetric matrix. Put
\[
\bu=(\by^\top,\boldeta^\top)^\top,
\qquad
\bN(\bu)=
\binom{\bF(\by,\boldeta)}{\bG(\by,\boldeta)}
=\bigl(N_1(\bu),\ldots,N_6(\bu)\bigr)^\top.
\]
Set
\[
A_0=\diag(0_{4\times4},\Omega),
\qquad
\mE(\tau)=e^{\tau A_0/\eps}.
\]
Then \eqref{eq:transformed} can be written as
\[
\dot{\bu}
=
\frac1\eps A_0\bu+\bN(\bu).
\]
Indeed, the only \(\eps^{-1}\)-term in \eqref{eq:transformed} is
\(\eps^{-1}\Omega\boldeta\), and therefore the singular matrix is the
constant matrix \(A_0\). It does not depend on \(t\), \(n\), or \(\bu_n\). For \(\balpha\in\cI_6^k\), consider the monomial
\(\chi_{\balpha}(\bu)=\bu^{\balpha}\). The contribution of the singular part
\(\dot{\bu}=\eps^{-1}A_0\bu\) to its derivative is
\[
\frac{d}{dt}\bu^{\balpha}
=
\frac1\eps
\sum_{j=1}^6
\alpha_j
\bu^{\balpha-\be_j}
(A_0\bu)_j
=
\frac1\eps
\sum_{j=1}^6\sum_{\ell=1}^6
\alpha_j(A_0)_{j\ell}
\bu^{\balpha-\be_j+\be_\ell}.
\]
Thus, if
\[
\bu^{[k]}=
\bigl(\bu^{\balpha}\bigr)_{\balpha\in\cI_6^k},
\]
then the singular part of the monomial system has the form
\begin{equation}\label{eq:A0-extension-def}
\frac{d}{dt}\bu^{[k]}
=
\frac1\eps A_0^{[k]}\bu^{[k]}.
\end{equation}

Thanks to form \eqref{eq:y-z} the fast part \(\dot{\boldeta}=\eps^{-1}\Omega\boldeta\) can be integrated exactly, while the remaining terms \(\bF\) and \(\bG\) are uniformly smooth on compact sets as we establish in the following.

\begin{lemma}[Uniform smoothness]\label{lem:smooth}
If \(\bE\) and \(\bB_1\) are \(C^{k+1}\) and the exact and numerical solutions remain in a compact set, then all derivatives up to order \(k+1\) of \(\bF\) and \(\bG\) are uniformly bounded with respect to \(\eps\).
\end{lemma}

\begin{proof}
Let \(K\subset\R^6\) be a compact set containing all transformed exact and numerical values considered in the argument. Write
\[
\bq(\by,\boldeta)=\be_0\sigma+Q_\perp\bY+\eps Q_\perp\boldeta,\qquad
\bv(\by,\boldeta)=\nu\be_0+Q_\perp\Omega\boldeta .
\]
The Jacobian matrices of these affine maps are
\[
D_{\sigma}\bq=\be_0,\quad D_{\bY}\bq=Q_\perp,\quad
D_{\boldeta}\bq=\eps Q_\perp,\quad
D_{\nu}\bv=\be_0,\quad D_{\boldeta}\bv=Q_\perp\Omega ,
\]
and all second and higher derivatives of \(\bq\) and \(\bv\) vanish. Hence every derivative of \(\bq\) and \(\bv\) is bounded uniformly for \(0<\eps\le1\). More importantly, differentiation with respect to \(\boldeta\) never produces a negative power of \(\eps\); for \(\bq\) it produces the factor \(\eps Q_\perp\), and for \(\bv\) it produces the constant matrix \(Q_\perp\Omega\).

The function \(\widehat{\bH}\) appearing in \(\bF\) and \(\bG\) is
\[
\widehat{\bH}(\by,\boldeta)
=\bE(\bq(\by,\boldeta))
+\bv(\by,\boldeta)\times\bB_1(\bq(\by,\boldeta)).
\]
For a first derivative in any coordinate \(z\in\{\sigma,\nu,Y_1,Y_2,\eta_1,\eta_2\}\), the chain rule gives
\[
\partial_z\widehat{\bH}
=D\bE(\bq)\partial_z\bq
+\partial_z\bv\times\bB_1(\bq)
+\bv\times D\bB_1(\bq)\partial_z\bq .
\]
Every factor on the right-hand side is bounded on \(K\) uniformly in \(\eps\). For higher derivatives, the multivariate Fa\`a di Bruno formula reduces to finite sums involving derivatives of \(\bE\) and \(\bB_1\) up to the same order, multiplied by products of first derivatives of the affine maps \(\bq\) and \(\bv\). Since the affine maps have no higher derivatives and their first derivatives are uniformly bounded, for each multi-index \(\balpha\) with \(|\balpha|\le k+1\) there exists \(C_{\alpha}\), independent of \(\eps\), such that
\[
\sup_{(\by,\boldeta)\in K}\norm{\partial^{\alpha}\widehat{\bH}(\by,\boldeta)}\le C_{\alpha} .
\]
Finally,
\[
\bF(\by,\boldeta)=
\begin{pmatrix}
\nu\\
\be_0^\top\widehat{\bH}(\by,\boldeta)\\
-\eps\Omega^{-1}Q_\perp^\top\widehat{\bH}(\by,\boldeta)
\end{pmatrix},
\qquad
\bG(\by,\boldeta)=\Omega^{-1}Q_\perp^\top\widehat{\bH}(\by,\boldeta).
\]
Multiplication by the fixed matrices \(\be_0^\top\), \(\Omega^{-1}Q_\perp^\top\), and \(Q_\perp\) preserves the bounds, and the extra factor \(\eps\) in the \(\bY\)-equation only improves them. Therefore all derivatives of \(\bF\) and \(\bG\) up to order \(k+1\) are bounded uniformly with respect to \(\eps\).
\end{proof}

The following lemma guarantees that the matrix exponential $e^{\eps^{-1}A_0^{[k]}}$ does not grow rapidly as $\eps \to 0$, which would otherwise affect the uniform error estimates.

\begin{lemma}
\label{lem:rotation-extension}
Let \(A_0^{[k]}\) be the monomial extension defined by
\eqref{eq:A0-extension-def}. Then
\begin{equation}\label{eq:A0-extension-bound}
\sup_{0<\eps\le1}\sup_{\tau\in\R}
\norm{\exp(\tau A_0^{[k]}/\eps)}
\le C_k .
\end{equation}
\end{lemma}

\begin{proof}
The spectrum of \(A_0\) is \(\sigma(A_0)=\{0,0,0,0,i\omega,-i\omega\}\). Since \(A_0\) is normal, it is unitarily diagonalizable. The induced action on
the monomial vector \(\bu^{[k]}\) is block diagonal with respect to total
degree. For each \(\balpha\in\mathcal I_6^k\), the corresponding eigenvalue of
\(A_0^{[k]}\) has the form
\[
\lambda_{\balpha}
=
\sum_{j=1}^6 \alpha_j\lambda_j,
\qquad
\lambda_j\in\sigma(A_0).
\]
Hence \(\lambda_{\balpha}\in i\R\). Therefore \(A_0^{[k]}\) is diagonalizable
over \(\mathbb C\) with purely imaginary eigenvalues, and in a suitable
diagonalizing norm,
\[
\norm{\exp(\tau A_0^{[k]}/\eps)}\le1,
\qquad
\tau\in\R,\quad 0<\eps\le1 .
\]
Since the extended space is finite dimensional, this norm is equivalent to the
fixed Euclidean norm. Thus there exists \(C_k>0\), independent of \(\eps\) and
\(\tau\), such that
\[
\norm{\exp(\tau A_0^{[k]}/\eps)}\le C_k .
\]
\end{proof}

\section{The Small-Step Case}\label{sec:small}

The condition below means that the phase increment over one step,
\(h/\eps\), is bounded. The homogeneous rotation can therefore be separated
from the local change before the Taylor expansion is formed. Instead of
measuring the solution from the fixed point \(\bu(t_n)\), we use the reference
position \(\mE(t-t_n)\bu(t_n)\). The resulting deviation is \(O(h)\), so the
degree-\((k-1)\) Taylor polynomial has an \(O(h^k)\) remainder.

Assume
\begin{equation}\label{eq:small-step}
h\le c_0\eps.
\end{equation}
Recall \(\dot{\bu}=\eps^{-1}A_0\bu+\bN(\bu)\) and
\(\mE(\tau)=e^{\tau A_0/\eps}\). Fix a step \(t_n\to t_{n+1}\).
For the exact solution, define
\[
\bxi_n(t)=\bu(t)-\mE(t-t_n)\bu(t_n).
\]
The small-step CREI\((k)\) extended system is obtained by Taylor expanding
\(\bN(\mE(t-t_n)\bu(t_n)+\bxi_n)\) in the variable \(\bxi_n\) up to degree
\(k-1\).

The expansion variable in case \eqref{eq:small-step} is
\(\bxi_n(t)=\bu(t)-\mE(t-t_n)\bu(t_n)\). Before estimating the Taylor
remainder, we need a bound on \(\bxi_n\).
\begin{lemma}[Small-step deviation]\label{lem:small-dev}
Under assumption \eqref{eq:small-step}, we have
\[
\sup_{t_n\le t\le t_{n+1}}\norm{\bxi_n(t)}\le Ch.
\]
\end{lemma}

\begin{proof}
Put \(\tau=t-t_n\). Since
\[
\mE(\tau)\bu(t_n)
=\bigl(\by(t_n),e^{\tau\Omega/\eps}\boldeta(t_n)\bigr),
\]
the deviation is
\[
\bxi_n(t)=
\binom{\by(t)-\by(t_n)}
{\boldeta(t)-e^{\tau\Omega/\eps}\boldeta(t_n)} .
\]
The \(\by\)-equation contains no \(\eps^{-1}\) term, so
\[
\by(t)-\by(t_n)=\int_{t_n}^{t}\bF(\by(s),\boldeta(s))\,ds,
\]
and Lemma \ref{lem:smooth} gives
\[
\norm{\by(t)-\by(t_n)}
\le \int_{t_n}^{t}\norm{\bF(\by(s),\boldeta(s))}\,ds
\le C(t-t_n)\le Ch .
\]
For \(\boldeta\), the equation
\[
\dot{\boldeta}-\eps^{-1}\Omega\boldeta=\bG(\by,\boldeta)
\]
is integrated exactly by the variation-of-constants formula:
\[
\boldeta(t)-e^{(t-t_n)\Omega/\eps}\boldeta(t_n)
=\int_{t_n}^{t}e^{(t-s)\Omega/\eps}\bG(\by(s),\boldeta(s))\,ds.
\]
Because \(\Omega\) is skew-symmetric, \(e^{(t-s)\Omega/\eps}\) is orthogonal, and hence
\[
\norm{\boldeta(t)-e^{(t-t_n)\Omega/\eps}\boldeta(t_n)}
\le \int_{t_n}^{t}\norm{\bG(\by(s),\boldeta(s))}\,ds
\le Ch .
\]
Combining the two component estimates gives the assertion.
\end{proof}

The truncated extended system differs from the exact monomial system only by
the Taylor remainder of \(\bN(\mE(t-t_n)\bu(t_n)+\bxi_n)\). The next estimate
turns Lemma \ref{lem:small-dev} into the uniform error bound
\(\bR_n^{[k]}=O(h^k)\), which is the forcing term in the extended error
equation.
\begin{lemma}[Small-step Taylor remainder]\label{lem:small-rem}
The Taylor remainder in the small-step truncated extended system satisfies
\[
\sup_{t_n\le t\le t_{n+1}}\norm{\bR_n^{[k]}(t)}\le Ch^k.
\]
\end{lemma}

\begin{proof}
For each component \(N_\ell\) of \(\bN\), Taylor's
formula with integral remainder at \(\mE(t-t_n)\bu(t_n)\) gives
\[
N_\ell(\mE(t-t_n)\bu(t_n)+\bxi_n(t))
=\sum_{|\balpha|\le k-1}
\frac{\partial^{\balpha} N_\ell(\mE(t-t_n)\bu(t_n))}{\balpha!}
\bxi_n(t)^{\balpha}
+R_{\ell,n}^{(k)}(t),
\]
where
\[
R_{\ell,n}^{(k)}(t)
=\sum_{|\balpha|=k}\frac{k}{\balpha!}
\int_0^1(1-\theta)^{k-1}
\partial^{\balpha} N_\ell(\mE(t-t_n)\bu(t_n)+\theta\bxi_n(t))\,d\theta\,
\bxi_n(t)^{\balpha} .
\]
Set \(\bR_n^{[k]}(t)=(R_{1,n}^{(k)}(t),\ldots,R_{6,n}^{(k)}(t))^\top\) for
the degree-one remainder, and use the same notation for the finite vector of
remainders generated in the retained monomial equations.
By Lemma \ref{lem:smooth},
\[
\left|
\frac{k}{\balpha!}
\int_0^1(1-\theta)^{k-1}
\partial^{\balpha} N_\ell(\mE(t-t_n)\bu(t_n)+\theta\bxi_n(t))\,d\theta
\right|\le C_{\balpha} ,
\]
uniformly in \(n,t,h,\eps\). Lemma \ref{lem:small-dev} gives
\[
|\bxi_n(t)^{\balpha}|
\le \norm{\bxi_n(t)}^{|\balpha|}
\le Ch^k,\qquad |\balpha|=k.
\]
Thus every component \(R_{\ell,n}^{(k)}(t)\) is bounded by \(Ch^k\).

It remains to check that the same order is inherited by the extended monomial system. If \(\chi_{\bbeta}(t)=\bxi_n(t)^{\bbeta}\), \(|\bbeta|\le k-1\), then
\[
\frac{d}{dt}\chi_{\bbeta}
=\sum_{j=1}^6\beta_j\bxi_n^{\bbeta-\be_j}\dot\xi_{n,j}.
\]
Replacing \(\bN(\mE(t-t_n)\bu(t_n)+\bxi_n)\) in \(\dot\bxi_n\) by its Taylor
polynomial of degree \(k-1\) leaves the defect
\[
\sum_{j=1}^6\beta_j\bxi_n^{\bbeta-\be_j}R_{j,n}^{(k)}(t).
\]
For \(|\bbeta|\le k-1\), the factor \(\bxi_n^{\bbeta-\be_j}\) is bounded on the compact set, and \(R_{j,n}^{(k)}(t)=O(h^k)\). Since the number of monomials is finite, the full extended remainder satisfies
\[
\sup_{t_n\le t\le t_{n+1}}\norm{\bR_n^{[k]}(t)}\le Ch^k .
\]
\end{proof}

Next we prove the local error result in the small-step case.
\begin{theorem}[Small-step local error]\label{thm:small-local}
One CREI\((k)\) step starting from the exact value satisfies
\[
\norm{\widetilde{\bu}_{n+1}-\bu(t_{n+1})}\le Ch^{k+1}.
\]
Here
\[
\widetilde{\bu}_{n+1}
=\mE(h)\bu(t_n)+\Pi_2^7\widetilde{\bXi}_n^{[k-1]}(t_{n+1})
\]
is the value produced by the truncated small-step extended system initialized
at the exact value \(\bu(t_n)\).
\end{theorem}

\begin{proof}
Let
\[
\bXi_n^{[k-1]}(t)
=\bigl(\bxi_n(t)^{\bbeta}\bigr)_{\bbeta\in\cI_6^{k-1}}
\]
be the exact monomial vector and let
\(\widetilde{\bXi}_n^{[k-1]}(t)\) be the solution of the truncated extended
system with the same initial value
\((1,0,\ldots,0)^\top\). Define
\[
\boldsymbol{\mathcal D}_n^{[k-1]}(t)
=\bXi_n^{[k-1]}(t)-\widetilde{\bXi}_n^{[k-1]}(t).
\]
The symbol \(\boldsymbol{\mathcal D}_n^{[k-1]}\) denotes the difference of
the exact and truncated monomial vectors; it is distinct from the electric
field \(\bE\).
Then
\[
\frac{d}{dt}\boldsymbol{\mathcal D}_n^{[k-1]}(t)
=\widehat M_n^{[k-1]}(t)\boldsymbol{\mathcal D}_n^{[k-1]}(t)
+\bR_n^{[k]}(t),
\qquad
\boldsymbol{\mathcal D}_n^{[k-1]}(t_n)=\bzero.
\]
Here \(\widehat M_n^{[k-1]}(t)\) is precisely the CREI extension matrix from
Section \ref{sec:crei}, evaluated along \(\mE(t-t_n)\bu(t_n)\) and truncated
at degree \(k-1\). It has the decomposition
\[
\widehat M_n^{[k-1]}(t)=\eps^{-1}A_0^{[k-1]}+K_n(t),
\]
where \(K_n(t)\) contains the Taylor coefficients
\(\partial^{\balpha}\bN(\mE(t-t_n)\bu(t_n))/\balpha!\), \(|\balpha|\le k-1\),
and the couplings produced by differentiating the retained monomials.
Let \(\bPsi_n(t,s)\) be the state-transition matrix of
\(\dot{\bw}=\widehat M_n^{[k-1]}(t)\bw\), namely
\[
\partial_t\bPsi_n(t,s)=\widehat M_n^{[k-1]}(t)\bPsi_n(t,s),
\qquad \bPsi_n(s,s)=I.
\]
To bound \(\bPsi_n\), factor out the singular oscillation:
\[
\bPsi_n(t,s)=e^{(t-s)A_0^{[k-1]}/\eps}V_n(t,s).
\]
Then \(V_n(s,s)=I\) and
\[
\partial_t V_n(t,s)=
e^{-(t-s)A_0^{[k-1]}/\eps}
K_n(t)
e^{(t-s)A_0^{[k-1]}/\eps}V_n(t,s).
\]
By Lemma \ref{lem:rotation-extension} and Lemma \ref{lem:smooth}, the coefficient in this equation has norm at most \(C\). Hence
\[
\norm{V_n(t,s)}
\le 1+\int_s^t C\norm{V_n(r,s)}\,dr .
\]
Gronwall's inequality yields
\[
\norm{V_n(t,s)}\le e^{C(t-s)}\le e^{Ch}\le C
\]
on one step. Multiplying by the uniformly bounded oscillatory exponential gives
\[
\sup_{t_n\le s\le t\le t_{n+1}}\norm{\bPsi_n(t,s)}\le C.
\]
By variation of constants,
\[
\boldsymbol{\mathcal D}_n^{[k-1]}(t_{n+1})
=\int_{t_n}^{t_{n+1}}\bPsi_n(t_{n+1},s)\bR_n^{[k]}(s)\,ds.
\]
Using Lemma \ref{lem:small-rem},
\[
\norm{\boldsymbol{\mathcal D}_n^{[k-1]}(t_{n+1})}
\le \int_{t_n}^{t_{n+1}}
\norm{\bPsi_n(t_{n+1},s)}\norm{\bR_n^{[k]}(s)}\,ds
\le C\int_{t_n}^{t_{n+1}}h^k\,ds
\le Ch^{k+1}.
\]
Applying the projection \(\Pi_2^7\) gives the asserted local error.
\end{proof}

The next lemma gives the stability estimate used in the global error recursion.
\begin{lemma}\label{lem:small-stab}
Let \(\bu_{n+1}^{(j)}\) be the CREI\((k)\) value obtained from
\(\bu_n^{(j)}\), \(j=1,2\). Then
\[
\norm{\bu_{n+1}^{(1)}-\bu_{n+1}^{(2)}}
\le (1+Ch)\norm{\bu_n^{(1)}-\bu_n^{(2)}}.
\]
\end{lemma}

\begin{proof}
For \(j=1,2\), set
\[
\widehat{\bu}^{(j)}(t)=\mE(t-t_n)\bu_n^{(j)}
\]
and let \(\widetilde{\bXi}^{(j)}(t)\) be the monomial vector produced by the
small-step truncated extended system built at \(\widehat{\bu}^{(j)}(t)\).
The projected numerical state is
\[
\widetilde{\bu}^{(j)}(t)
=\widehat{\bu}^{(j)}(t)+\Pi_2^7\widetilde{\bXi}^{(j)}(t).
\]
In particular, \(\bu_{n+1}^{(j)}=\widetilde{\bu}^{(j)}(t_{n+1})\).
For a retained monomial vector \(\bchi=(\chi_{\balpha})_{|\balpha|\le k-1}\),
define
\[
\mathcal P^{(j)}(t,\bchi)
=
\sum_{|\balpha|\le k-1}
\frac{\partial^{\balpha}\bN(\widehat{\bu}^{(j)}(t))}{\balpha!}
\chi_{\balpha}.
\]
Then the projected equations can be written as
\[
\frac{d}{dt}\widetilde{\bu}^{(j)}(t)
=\eps^{-1}A_0\widetilde{\bu}^{(j)}(t)
+\mathcal P^{(j)}\bigl(t,\widetilde{\bXi}^{(j)}(t)\bigr),
\qquad j=1,2,
\]
On the compact set, the coefficients
\(\partial^{\balpha}\bN(\widehat{\bu}^{(j)}(t))/\balpha!\) are bounded and
Lipschitz with respect to the initial state. Moreover, the monomial vector
generated by two initial states satisfies the finite-dimensional bound
\[
\norm{\widetilde{\bXi}^{(1)}(s)-\widetilde{\bXi}^{(2)}(s)}
\le C\norm{\bu_n^{(1)}-\bu_n^{(2)}},\qquad t_n\le s\le t_{n+1},
\]
which follows from the mean-value theorem and the uniform boundedness of the extended propagators on one step.

Variation of constants for the projected degree-one block gives
\[
\begin{aligned}
\widetilde{\bu}^{(1)}(t_{n+1})-\widetilde{\bu}^{(2)}(t_{n+1})
&=e^{hA_0/\eps}\bigl(\bu_n^{(1)}-\bu_n^{(2)}\bigr)\\
&\quad+\int_{t_n}^{t_{n+1}}e^{(t_{n+1}-s)A_0/\eps}
\left[
\mathcal P^{(1)}\bigl(s,\widetilde{\bXi}^{(1)}(s)\bigr)
-\mathcal P^{(2)}\bigl(s,\widetilde{\bXi}^{(2)}(s)\bigr)
\right]\,ds .
\end{aligned}
\]
Since \(A_0=\diag(0,\Omega)\) is skew-symmetric on the oscillatory block, \(e^{hA_0/\eps}\) has norm one on the projected block. The integrand is bounded by \(C\norm{\bu_n^{(1)}-\bu_n^{(2)}}\). Consequently
\[
\norm{\bu_{n+1}^{(1)}-\bu_{n+1}^{(2)}}
\le \norm{\bu_n^{(1)}-\bu_n^{(2)}}+
Ch\norm{\bu_n^{(1)}-\bu_n^{(2)}},
\]
which is the asserted estimate.
\end{proof}

By combining Lemmas \ref{lem:small-dev}, \ref{lem:small-rem}, and
\ref{lem:small-stab} with Theorem \ref{thm:small-local}, we obtain the uniform
convergence result for the small-step case.
\begin{theorem}[Small-step global error]\label{thm:small-global}
If \(h\le c_0\eps\), then
\[
\max_{0\le n\le T/h}\norm{\bu_n-\bu(t_n)}\le Ch^k.
\]
\end{theorem}

\begin{proof}
Let \(\be_n=\bu_n-\bu(t_n)\). By the stability estimate and the local error bound,
\[
\norm{\be_{n+1}}\le (1+Ch)\norm{\be_n}+Ch^{k+1}.
\]
Set \(a_n=\norm{\be_n}\), \(\alpha=Ch\), and \(\beta=Ch^{k+1}\). Then
\[
a_{n+1}\le (1+\alpha)a_n+\beta,\qquad a_0=0.
\]
We first record the elementary discrete Gronwall calculation. By induction,
\[
a_n\le (1+\alpha)^n a_0+\beta\sum_{j=0}^{n-1}(1+\alpha)^j .
\]
Indeed, multiplying the bound for \(a_n\) by \(1+\alpha\) and adding \(\beta\) gives the bound for \(a_{n+1}\). Since \(a_0=0\), this yields
\[
\norm{\be_n}
\le Ch^{k+1}\sum_{j=0}^{n-1}(1+Ch)^j .
\]
For \(nh\le T\),
\[
\sum_{j=0}^{n-1}(1+Ch)^j
\le n(1+Ch)^n
\le \frac{T}{h}e^{CT}.
\]
Consequently,
\[
\norm{\be_n}\le Ch^k,\qquad nh\le T.
\]
\end{proof}

\section{The Large-Step Case}\label{sec:large}

Here \(h/\eps\) is bounded from below and may be large. The homogeneous
solution can travel an \(O(1)\) distance from the fixed point \(\bu_n\) during
one step, since
\((e^{(t-t_n)A_0/\eps}-I)\bu_n\) is not controlled by \(h\) uniformly in
\(\eps\). A Taylor expansion at \(\bu_n\) would therefore mix the known rapid
rotation with the genuinely local correction. CREI instead follows the
reference position \(\widehat{\bu}_n(t)=\mE(t-t_n)\bu_n\) and expands only in
\(\bdelta_n=\bu-\widehat{\bu}_n\). This keeps \(\bdelta_n=O(h)\) and yields the
same power of \(h\) in the Taylor remainder independently of \(h/\eps\).

Assume
\begin{equation}\label{eq:large-step}
h\ge c_1\eps.
\end{equation}
For a step starting from \(\bu_n=(\by_n,\boldeta_n)\), use the homogeneous reference curve
\[
\widehat{\by}_n(t)=\by_n,\qquad
\widehat{\boldeta}_n(t)=e^{(t-t_n)\Omega/\eps}\boldeta_n.
\]
Define \(\delta\by(t)=\by(t)-\widehat{\by}_n(t)\),
\(\delta\boldeta(t)=\boldeta(t)-\widehat{\boldeta}_n(t)\), and
\(\bdelta_n(t)=(\delta\by(t)^\top,\delta\boldeta(t)^\top)^\top\). CREI\((k)\)
is obtained by Taylor expanding
\(\bN(\widehat{\bu}_n(t)+\bdelta_n)\) in \(\bdelta_n\) up to degree \(k-1\).

The expansion variable in this case is
\(\bdelta_n=(\delta\by^\top,\delta\boldeta^\top)^\top\). The estimate below
shows that the degree-\(k\) terms in the Taylor expansion of
\(\bN(\widehat{\bu}_n(t)+\bdelta_n)\) are \(O(h^k)\), even when the reference
curve contains \(e^{(t-t_n)\Omega/\eps}\boldeta_n\).
\begin{lemma}[Large-step deviation]\label{lem:large-dev}
Along the exact solution,
\[
\sup_{t_n\le t\le t_{n+1}}\bigl(\norm{\delta\by(t)}+\norm{\delta\boldeta(t)}\bigr)\le Ch.
\]
\end{lemma}

\begin{proof}
The \(\by\)-part satisfies, directly from \(\dot{\by}=\bF(\by,\boldeta)\),
\[
\delta\by(t)=\int_{t_n}^{t}\bF(\by(s),\boldeta(s))\,ds,
\]
and therefore
\[
\norm{\delta\by(t)}
\le \int_{t_n}^{t}\norm{\bF(\by(s),\boldeta(s))}\,ds
\le Ch .
\]
For the \(\boldeta\)-part, subtract the homogeneous equation
\[
\frac{d}{dt}\widehat{\boldeta}_n(t)=\eps^{-1}\Omega\widehat{\boldeta}_n(t)
\]
from \(\dot{\boldeta}=\eps^{-1}\Omega\boldeta+\bG(\by,\boldeta)\). Since
\(\delta\boldeta(t_n)=\bzero\), variation of constants gives
\[
\delta\boldeta(t)=\int_{t_n}^{t}e^{(t-s)\Omega/\eps}\bG(\by(s),\boldeta(s))\,ds.
\]
The propagator is orthogonal and \(\bG\) is bounded on the compact set. Hence
\[
\norm{\delta\boldeta(t)}
\le \int_{t_n}^{t}\norm{\bG(\by(s),\boldeta(s))}\,ds
\le Ch .
\]
\end{proof}

The next estimate identifies the defect caused by replacing
\(\bN(\widehat{\bu}_n(t)+\bdelta_n)\) by its Taylor polynomial of degree
\(k-1\). It provides the bound \(\bR_n^{[k]}=O(h^k)\) for the large-step
extended error equation.
\begin{theorem}[Large-step Taylor remainder]\label{thm:large-rem}
The Taylor remainder in the large-step truncated extended system satisfies
\[
\sup_{t_n\le t\le t_{n+1}}\norm{\bR_n^{[k]}(t)}\le Ch^k.
\]
\end{theorem}

\begin{proof}
Let
\[
\widehat{\bu}_n(t)=
\binom{\widehat{\by}_n(t)}{\widehat{\boldeta}_n(t)}
=\binom{\by_n}{e^{(t-t_n)\Omega/\eps}\boldeta_n},
\qquad
\bdelta_n(t)=\bu(t)-\widehat{\bu}_n(t).
\]
For each component \(N_\ell\) of \(\bN=(N_1,\ldots,N_6)^\top\), Taylor's
formula at \(\widehat{\bu}_n(t)\) gives
\[
N_\ell(\widehat{\bu}_n(t)+\bdelta_n(t))
=\sum_{|\balpha|\le k-1}
\frac{\partial^{\balpha} N_\ell(\widehat{\bu}_n(t))}{\balpha!}
\bdelta_n(t)^{\balpha}
+R_{\ell,n}^{(k)}(t),
\]
where the integral remainder is
\[
R_{\ell,n}^{(k)}(t)
=\sum_{|\balpha|=k}\frac{k}{\balpha!}
\int_0^1(1-\theta)^{k-1}
\partial^{\balpha} N_\ell(\widehat{\bu}_n(t)+\theta\bdelta_n(t))\,d\theta\,
\bdelta_n(t)^{\balpha} .
\]
Set \(\bR_n^{[k]}(t)\) to be the finite vector of degree-one and retained
monomial remainders generated by \(R_{\ell,n}^{(k)}(t)\).
Lemma \ref{lem:smooth} bounds the derivatives in the integral uniformly in \(\eps\), and Lemma \ref{lem:large-dev} gives
\[
|\bdelta_n(t)^{\balpha}|\le \norm{\bdelta_n(t)}^{|\balpha|}\le Ch^k,\qquad |\balpha|=k .
\]
Therefore \(R_{\ell,n}^{(k)}(t)=O(h^k)\) uniformly on the step.

The extended system is obtained by differentiating monomials
\[
\bdelta_n^{\bbeta},\qquad |\bbeta|\le k-1.
\]
The defect in the derivative of such a monomial is
\[
\sum_{j=1}^6\beta_j\bdelta_n^{\bbeta-\be_j}R_{j,n}^{(k)}(t).
\]
Since the deviation \(\bdelta_n\) is bounded by \(Ch\) and, in particular, is bounded independently of \(h\), each of these terms is \(O(h^k)\). There are only finitely many multi-indices in \(\cI_6^{k-1}\), so the full extended remainder satisfies the asserted estimate.
\end{proof}

With the large-step remainder controlled, the exact deviation monomial vector
\(\bZ_n^{[k-1]}\) can be compared with
\(\widetilde{\bZ}_n^{[k-1]}\). The boundedness of
\(\exp((t-s)A_0^{[k-1]}/\eps)\) converts the \(O(h^k)\) forcing into an
\(O(h^{k+1})\) one-step error after projection by \(\Pi_2^7\).
\begin{lemma}[Large-step local error]\label{lem:large-local}
One CREI\((k)\) step starting from the exact value satisfies
\[
\norm{\widetilde{\bu}_{n+1}-\bu(t_{n+1})}\le Ch^{k+1}.
\]
Here
\[
\widetilde{\bu}_{n+1}
=\widehat{\bu}_n(t_{n+1})
+\Pi_2^7\widetilde{\bZ}_n^{[k-1]}(t_{n+1})
\]
is the value produced by the truncated large-step extended system initialized
at the exact value \(\bu(t_n)\).
\end{lemma}

\begin{proof}
Let
\[
\bZ_n^{[k-1]}(t)
=\bigl(\bdelta_n(t)^{\bbeta}\bigr)_{\bbeta\in\cI_6^{k-1}}
\]
be the exact deviation monomial vector, and let
\(\widetilde{\bZ}_n^{[k-1]}(t)\) be the solution of the truncated large-step
extended system with the same initial value \((1,0,\ldots,0)^\top\). Define
\[
\boldsymbol{\mathcal D}_n^{[k-1]}(t)
=\bZ_n^{[k-1]}(t)-\widetilde{\bZ}_n^{[k-1]}(t).
\]
Then
\[
\frac{d}{dt}\boldsymbol{\mathcal D}_n^{[k-1]}(t)
=\widehat M_n^{[k-1]}(t)\boldsymbol{\mathcal D}_n^{[k-1]}(t)
+\bR_n^{[k]}(t),
\qquad
\boldsymbol{\mathcal D}_n^{[k-1]}(t_n)=\bzero.
\]
The matrix \(\widehat M_n^{[k-1]}(t)\) is the CREI extension matrix obtained
from the Taylor polynomial of \(\bN(\widehat{\bu}_n(t)+\bdelta)\) through
degree \(k-1\). It has the decomposition
\[
\widehat M_n^{[k-1]}(t)=\eps^{-1}A_0^{[k-1]}+K_n(t),
\]
where \(K_n(t)\) contains the Taylor coefficients
\(\partial^{\balpha}\bN(\widehat{\bu}_n(t))/\balpha!\), \(|\balpha|\le k-1\),
and the couplings produced by differentiating the retained monomials. These
coefficients are evaluated at
\[
\widehat{\bu}_n(t)=\bigl(\by_n,e^{(t-t_n)\Omega/\eps}\boldeta_n\bigr).
\]
Since the exponential is orthogonal, this curve remains in a compact set whenever \(\bu_n\) does. Lemma \ref{lem:smooth} therefore gives
\[
\sup_{t_n\le t\le t_{n+1}}\norm{K_n(t)}\le C.
\]
Let \(\bPsi_n(t,s)\) be the state-transition matrix of
\(\dot{\bw}=\widehat M_n^{[k-1]}(t)\bw\).
As in the proof of Theorem \ref{thm:small-local}, write
\[
\bPsi_n(t,s)=e^{(t-s)A_0^{[k-1]}/\eps}V_n(t,s).
\]
Then
\[
\partial_t V_n(t,s)
=e^{-(t-s)A_0^{[k-1]}/\eps}
K_n(t)
e^{(t-s)A_0^{[k-1]}/\eps}V_n(t,s),
\qquad
V_n(s,s)=I .
\]
By Lemma \ref{lem:rotation-extension},
\[
\norm{\partial_t V_n(t,s)}
\le C\norm{V_n(t,s)}.
\]
Equivalently,
\[
\norm{V_n(t,s)}
\le 1+C\int_s^t\norm{V_n(r,s)}\,dr .
\]
Gronwall's inequality implies \(\norm{V_n(t,s)}\le e^{C(t-s)}\). Since \(t-s\le h\le T\), we obtain
\[
\norm{\bPsi_n(t,s)}\le C.
\]
Variation of constants gives
\[
\boldsymbol{\mathcal D}_n^{[k-1]}(t_{n+1})
=\int_{t_n}^{t_{n+1}}
\bPsi_n(t_{n+1},s)\bR_n^{[k]}(s)\,ds .
\]
Using Theorem \ref{thm:large-rem},
\[
\norm{\boldsymbol{\mathcal D}_n^{[k-1]}(t_{n+1})}
\le \int_{t_n}^{t_{n+1}}
\norm{\bPsi_n(t_{n+1},s)}
\norm{\bR_n^{[k]}(s)}\,ds
\le C\int_{t_n}^{t_{n+1}}h^k\,ds
\le Ch^{k+1}.
\]
The projection \(\Pi_2^7\) gives the local error for \(\bu\).
\end{proof}

For the global estimate, the numerical values produced from two initial states
must satisfy the following Lipschitz bound.
\begin{lemma}[Large-step stability]\label{lem:large-stab}
Let \(\bu_{n+1}^{(j)}\) be the CREI\((k)\) value obtained from
\(\bu_n^{(j)}\), \(j=1,2\). Then
\[
\norm{\bu_{n+1}^{(1)}-\bu_{n+1}^{(2)}}
\le (1+Ch)\norm{\bu_n^{(1)}-\bu_n^{(2)}}.
\]
\end{lemma}

\begin{proof}
The two reference curves are
\[
\widehat{\by}^{(j)}(t)=\by_n^{(j)},\qquad
\widehat{\boldeta}^{(j)}(t)=e^{(t-t_n)\Omega/\eps}\boldeta_n^{(j)},\qquad j=1,2.
\]
Write
\[
\widehat{\bu}^{(j)}(t)=
\bigl((\widehat{\by}^{(j)}(t))^\top,
(\widehat{\boldeta}^{(j)}(t))^\top\bigr)^\top .
\]
Since \(e^{(t-t_n)\Omega/\eps}\) is orthogonal, these curves are Lipschitz in the initial state uniformly in \(\eps\):
\[
\sup_{t_n\le t\le t_{n+1}}
\norm{\widehat{\bu}^{(1)}(t)-\widehat{\bu}^{(2)}(t)}
\le \norm{\bu_n^{(1)}-\bu_n^{(2)}}.
\]
Let \(\widetilde{\bZ}^{(j)}(t)\) be the retained monomial vector produced by
the \(j\)-th truncated large-step extended system.
Set
\[
\widetilde{\bu}^{(j)}(t)=
\widehat{\bu}^{(j)}(t)+\Pi_2^7\widetilde{\bZ}^{(j)}(t),
\qquad
\bu_{n+1}^{(j)}=\widetilde{\bu}^{(j)}(t_{n+1}).
\]
For a retained monomial vector
\(\bchi=(\chi_{\balpha})_{|\balpha|\le k-1}\), define
\[
\mathcal P^{(j)}(t,\bchi)
=
\sum_{|\balpha|\le k-1}
\frac{\partial^{\balpha}\bN(\widehat{\bu}^{(j)}(t))}{\balpha!}
\chi_{\balpha},
\qquad j=1,2 .
\]
The Taylor coefficients
\(\partial^{\balpha}\bN(\widehat{\bu}^{(j)}(t))/\balpha!\) are smooth
functions of the reference curves, and all relevant states remain in the same
compact set. Hence
\[
\norm{\mathcal P^{(1)}(s,\widetilde{\bZ}^{(1)}(s))
-\mathcal P^{(2)}(s,\widetilde{\bZ}^{(2)}(s))}
\le C\norm{\bu_n^{(1)}-\bu_n^{(2)}},
\qquad t_n\le s\le t_{n+1}.
\]
Here \(\widetilde{\bZ}^{(j)}\) denotes the retained monomial vector in the
\(j\)-th truncated large-step system; the bound follows from the mean-value
theorem and the uniform boundedness of the extended oscillatory propagator.

The projected degree-one block satisfies the same variation-of-constants identity as in the proof of Lemma \ref{lem:small-stab}:
\[
\begin{aligned}
\widetilde{\bu}^{(1)}(t_{n+1})-\widetilde{\bu}^{(2)}(t_{n+1})
&=e^{hA_0/\eps}\bigl(\bu_n^{(1)}-\bu_n^{(2)}\bigr)\\
&\quad+\int_{t_n}^{t_{n+1}}e^{(t_{n+1}-s)A_0/\eps}
\left[
\mathcal P^{(1)}(s,\widetilde{\bZ}^{(1)}(s))
-\mathcal P^{(2)}(s,\widetilde{\bZ}^{(2)}(s))
\right]\,ds .
\end{aligned}
\]
The first factor is norm preserving on the projected block, and the integral has length \(h\). Therefore
\[
\norm{\bu_{n+1}^{(1)}-\bu_{n+1}^{(2)}}
\le (1+Ch)\norm{\bu_n^{(1)}-\bu_n^{(2)}}.
\]
\end{proof}

The local and stability estimates imply
\(\norm{\be_{n+1}}\le(1+Ch)\norm{\be_n}+Ch^{k+1}\). The theorem below gives
the resulting \(O(h^k)\) mesh error.
\begin{theorem}[Large-step global error]\label{thm:large-global}
If \(h\ge c_1\eps\), then
\[
\max_{0\le n\le T/h}\norm{\bu_n-\bu(t_n)}\le Ch^k.
\]
\end{theorem}

\begin{proof}
Let
\[
\be_n=\bu_n-\bu(t_n)
\]
be the global error, and let \(\bu_{n+1}^{\star}\) be the CREI\((k)\) value
obtained by starting the same step from \(\bu(t_n)\). Then
\[
\norm{\be_{n+1}}
\le
\norm{\bu_{n+1}-\bu_{n+1}^{\star}}
+\norm{\bu_{n+1}^{\star}-\bu(t_{n+1})}.
\]
The first term is bounded by Lemma \ref{lem:large-stab}, and the second term is bounded by Lemma \ref{lem:large-local}. Therefore
\[
\norm{\be_{n+1}}
\le (1+Ch)\norm{\be_n}+Ch^{k+1}.
\]
Set \(a_n=\norm{\be_n}\). Since \(a_0=0\), repeated substitution gives
\[
a_n\le Ch^{k+1}\sum_{j=0}^{n-1}(1+Ch)^j .
\]
For \(nh\le T\),
\[
\sum_{j=0}^{n-1}(1+Ch)^j
\le n(1+Ch)^n
\le \frac{T}{h}e^{CT}.
\]
Thus \(a_n\le Ch^k\), uniformly for \(0<\eps\le1\) under
\eqref{eq:large-step}.
\end{proof}

Combining the estimates from Sections \ref{sec:small} and \ref{sec:large}
gives the following result.
\begin{theorem}\label{thm:main}
Fix \(T>0\) and an integer \(k\ge2\). Assume that \(\bE\) and \(\bB_1\) are
\(C^{k+1}\) on an open set containing all exact and numerical trajectories,
that the transformed exact and numerical solutions remain in a compact set
independent of \(\eps\), and that \(\bu_0=\bu(0)\). Let \(t_n=nh\) and
\(N=\lfloor T/h\rfloor\), and let \(\bu_n\) be produced by CREI\((k)\).
Then
\[
\max_{0\le n\le N}\norm{\bu_n-\bu(t_n)}\le Ch^k
\]
in both regimes \(h\le c_0\eps\) and \(h\ge c_1\eps\). Consequently,
\[
\max_{0\le n\le N}
\bigl(\norm{\bq_n-\bq(t_n)}+\norm{\bp_n-\bp(t_n)}\bigr)\le Ch^k.
\]
\end{theorem}

\begin{proof}
Combine Theorems \ref{thm:small-global} and \ref{thm:large-global} with the \(\eps\)-uniform Lipschitz property of the inverse transformation stated in Remark \ref{rem:inverse-lip}.
\end{proof}

The implementation is summarized below.
\begin{algorithm}
\caption{CREI for three-dimensional CPD}
\label{alg:multiscale}
\begin{algorithmic}[1]
\State Transform $(\bx_n, \bv_n)$ to $(\bq_n, \bp_n)$ by \eqref{eq:qp-transform}.
\State Transform $(\bq_n, \bp_n)$ to $\mathbf{u}_n = (\by_n, \boldsymbol{\eta}_n)$ using \eqref{eq:r-def} and \eqref{eq:new-vars}.
\State Set $\widehat{\bu}_n(t)=\mE(t-t_n)\bu_n$ and form the Taylor polynomial of $\bN$ along this curve.
\State Integrate the truncated monomial-extension system and project its degree-one block.
\State Reconstruct $(\bq_{n+1}, \bp_{n+1})$ by Proposition~\ref{prop:inverse}, and then get the $(\bx_{n+1}, \bv_{n+1})$.
\end{algorithmic}
\end{algorithm}

\section{Numerical Experiments}\label{sec:numerics}

This section tests the error bounds in Sections \ref{sec:small} and \ref{sec:large}. We first fix \(\eps\) and vary \(h\), and then fix \(h\) and vary \(\eps\). The parameter ranges follow the CPD-oriented tests in \cite{wangWangZhu2026}. As defined in Section \ref{sec:crei}, retaining monomials through degree \(k\) gives CREI\((k+1)\).

Corresponding to the system \eqref{eq:cpd}, we take
\[
\bB_0=(0,0,1/2)^\top,
\]
and use the smooth fields
\[
\bE(\bx)=
\begin{pmatrix}
-x_1(x_1^2+x_2^2+0.25)^{-3/2}\\
-x_2(x_1^2+x_2^2+0.25)^{-3/2}\\
0.15\sin x_3
\end{pmatrix},
\qquad
\bB_1(\bx)=
\begin{pmatrix}
0.12\cos x_2\\
0.10\sin(x_1+x_3)\\
0.08\cos(x_1-x_2)
\end{pmatrix}.
\]
The final time is \(T=1\). We use one rounded initial datum throughout the numerical section:
\[
\bx_0=(2.2,0.6,0.5)^\top,\qquad\bv_0=(0,1,0.7)^\top.
\]
This choice keeps every component to one decimal place and ensures that \(\bv_0\) is componentwise nonnegative. A high-accuracy reference solution is computed by an adaptive ODE solver with absolute and relative tolerances \(10^{-12}\), with the maximum step proportional to \(\eps\). Errors are measured in the transformed position-momentum variables:
\[
E_{\bq}(h)=\norm{\bq_N-\bq(T)},\qquad
E_{\bp}(h)=\norm{\bp_N-\bp(T)}.
\]

We compare CREI2, CREI3, and CREI4. CREI\(m\) retains all monomials \(\bdelta^{\balpha}\) with \(|\balpha|\le m-1\), where \(\bdelta=\bu-\widehat{\bu}_n(t)\). The reference curve is
\[
\widehat{\bu}_n(t)=\bigl(\by_n,e^{(t-t_n)\Omega/\eps}\boldeta_n\bigr),
\]
and the differentiated monomials are collected into the linear system defined in Section \ref{sec:crei}. The time-dependent Taylor coefficients are evaluated along \(\widehat{\bu}_n(t)\). No Richardson extrapolation is used.

In the first experiment, we fix \(\eps=\frac14\) and \(\eps=2^{-8}\), using
\(h=2^{-r},\ r=4,5,6,7\). Figures \ref{fig:fixed-eps-14} and \ref{fig:fixed-eps-28} show the position and momentum errors side by side. In each panel CREI2, CREI3, and CREI4 are plotted together, and the dotted guide lines indicate the reference slopes \(2,3,4\). The observed slopes agree well with the theoretical \(h^k\) behavior.

\begin{figure}[htbp]
\centering
\includegraphics[width=0.95\textwidth]{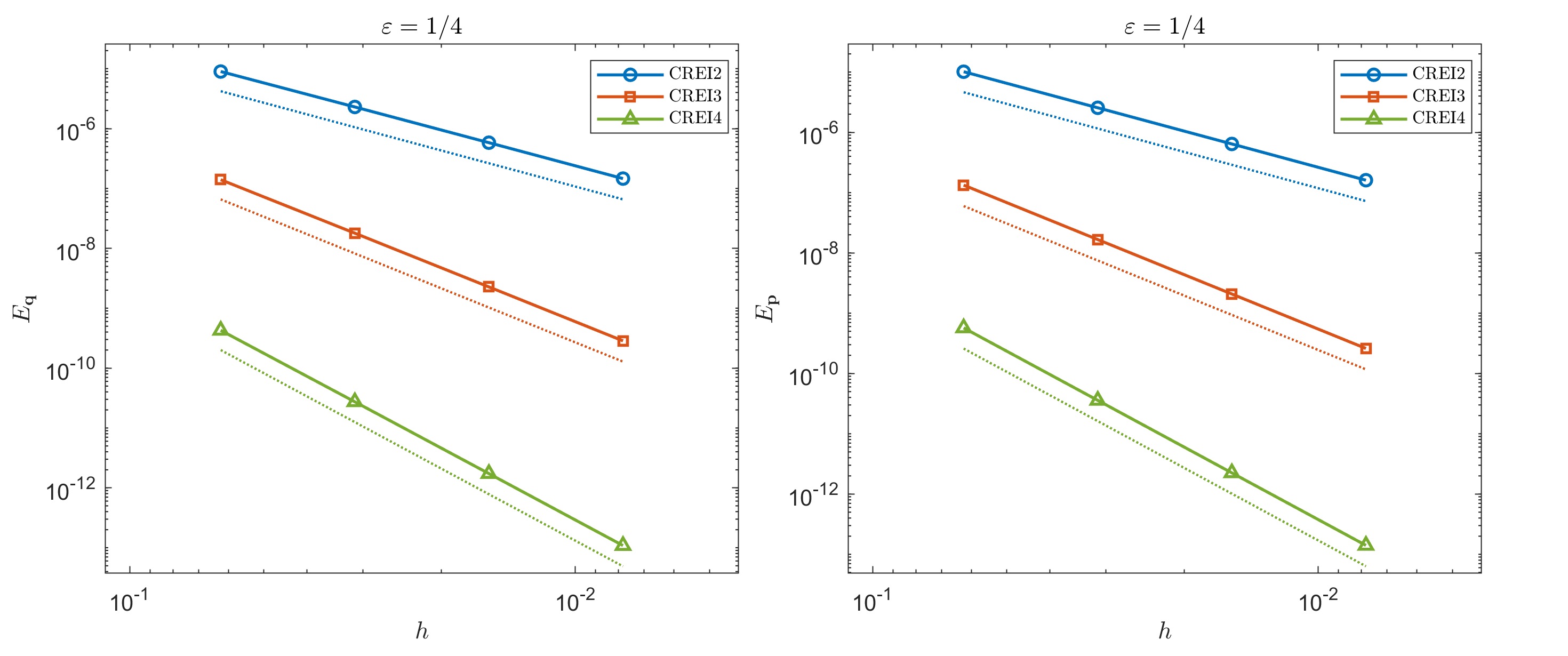}
\caption{Errors \(E_{\mathit{\bq}}\) and \(E_{\mathit{\bp}}\) versus \(h\) for fixed \(\eps=1/4\). The dotted lines are reference slopes \(O(h^2)\), \(O(h^3)\), and \(O(h^4)\).}
\label{fig:fixed-eps-14}
\end{figure}

\begin{figure}[htbp]
\centering
\includegraphics[width=0.95\textwidth]{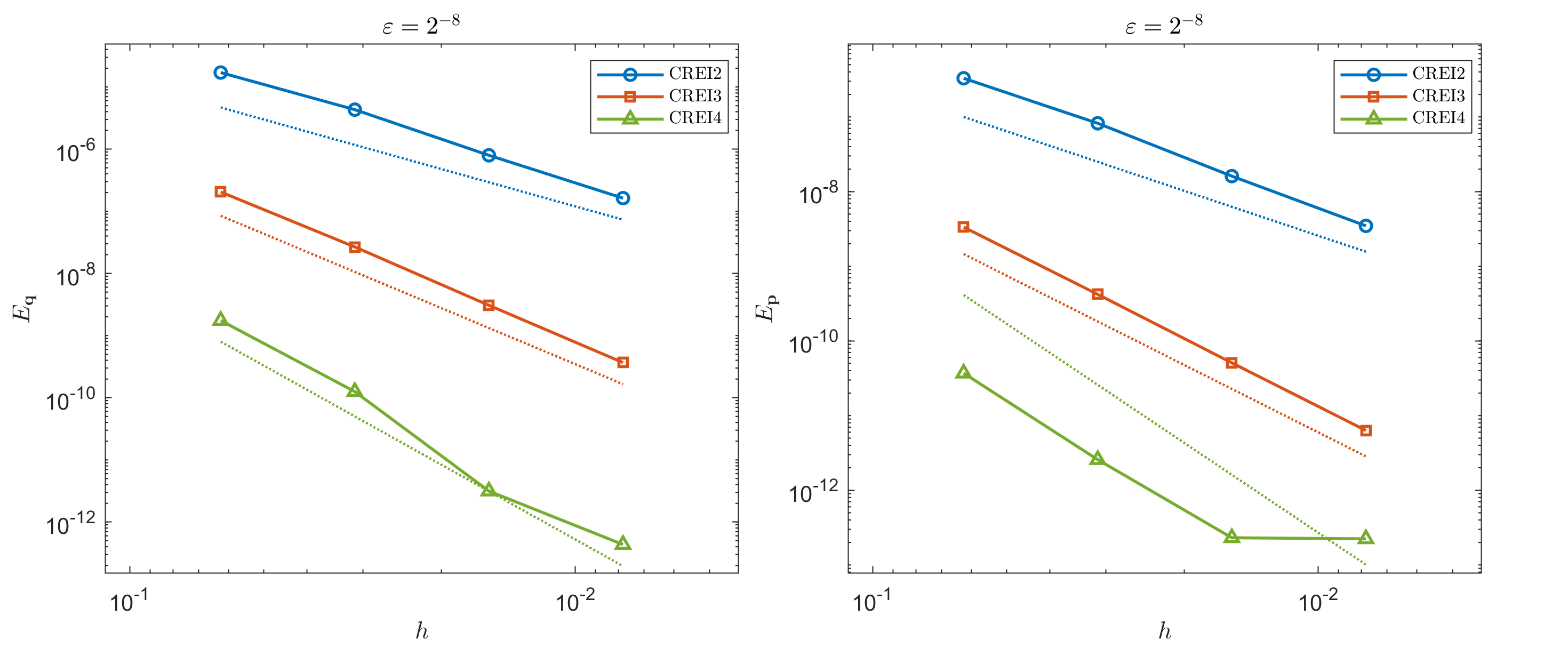}
\caption{Errors \(E_{\mathit{\bq}}\) and \(E_{\mathit{\bp}}\) versus \(h\) for fixed \(\eps=2^{-8}\). The dotted lines are reference slopes \(O(h^2)\), \(O(h^3)\), and \(O(h^4)\).}
\label{fig:fixed-eps-28}
\end{figure}

\FloatBarrier

In the second experiment, we test uniformity with respect to the oscillatory parameter. We take \(\eps=2^{-\ell},\  \ell=2,3,\ldots,8,\)
and use \(h=2^{-6}\) and \(h=1/2\). Figures \ref{fig:fixed-h-eq} and \ref{fig:fixed-h-ep} display \(E_{\mathit{\bq}}\) and \(E_{\mathit{\bp}}\) as functions of \(\eps\). The dotted horizontal guide lines mark the small-\(\eps\) error levels of the corresponding CREI methods. The errors remain bounded as \(\eps\) decreases, and increasing the CREI order reduces their magnitude.

\begin{figure}[htbp]
\centering
\includegraphics[width=0.95\textwidth]{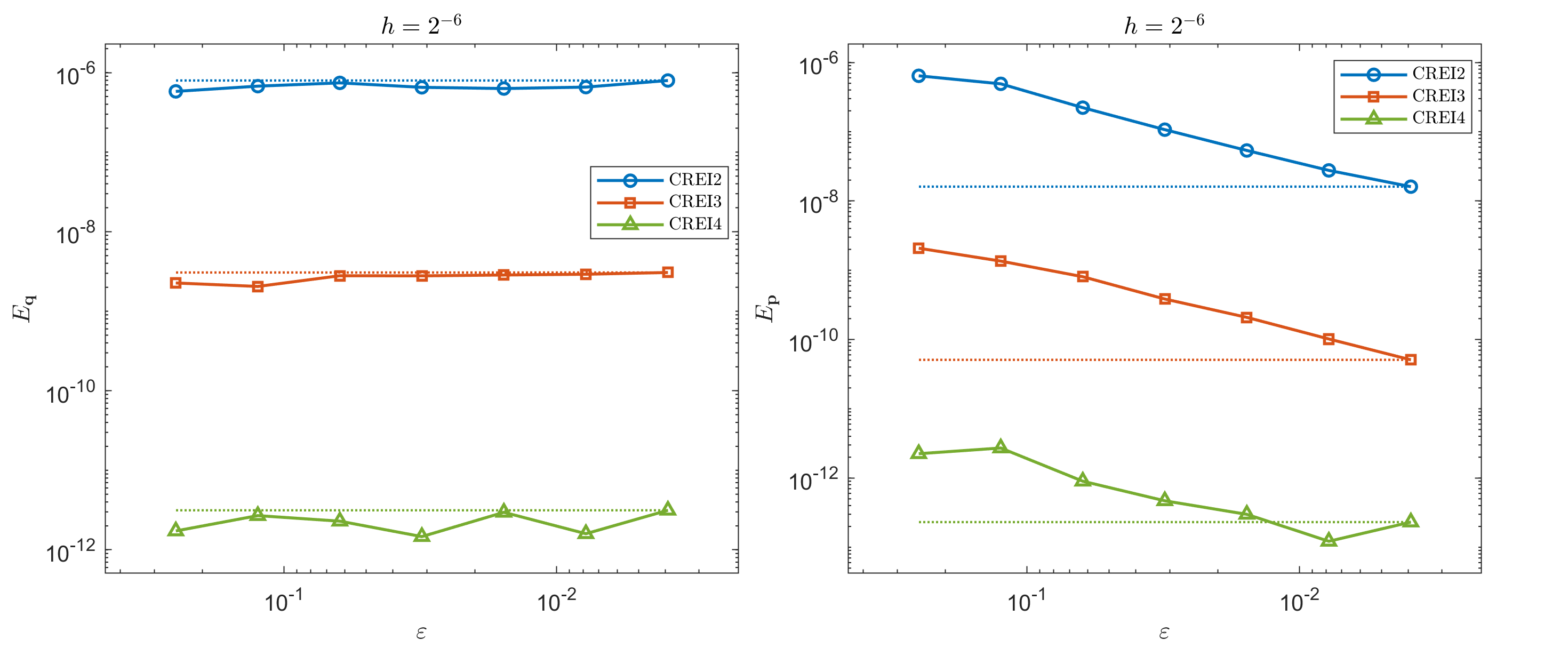}
\caption{Errors \(E_{\mathit{\bq}}\) and \(E_{\mathit{\bp}}\) versus \(\eps\) for fixed \(h=2^{-6}\). The dotted lines are horizontal references for the small-\(\eps\) error levels of CREI2, CREI3, and CREI4.}
\label{fig:fixed-h-eq}
\end{figure}

\begin{figure}[htbp]
\centering
\includegraphics[width=0.95\textwidth]{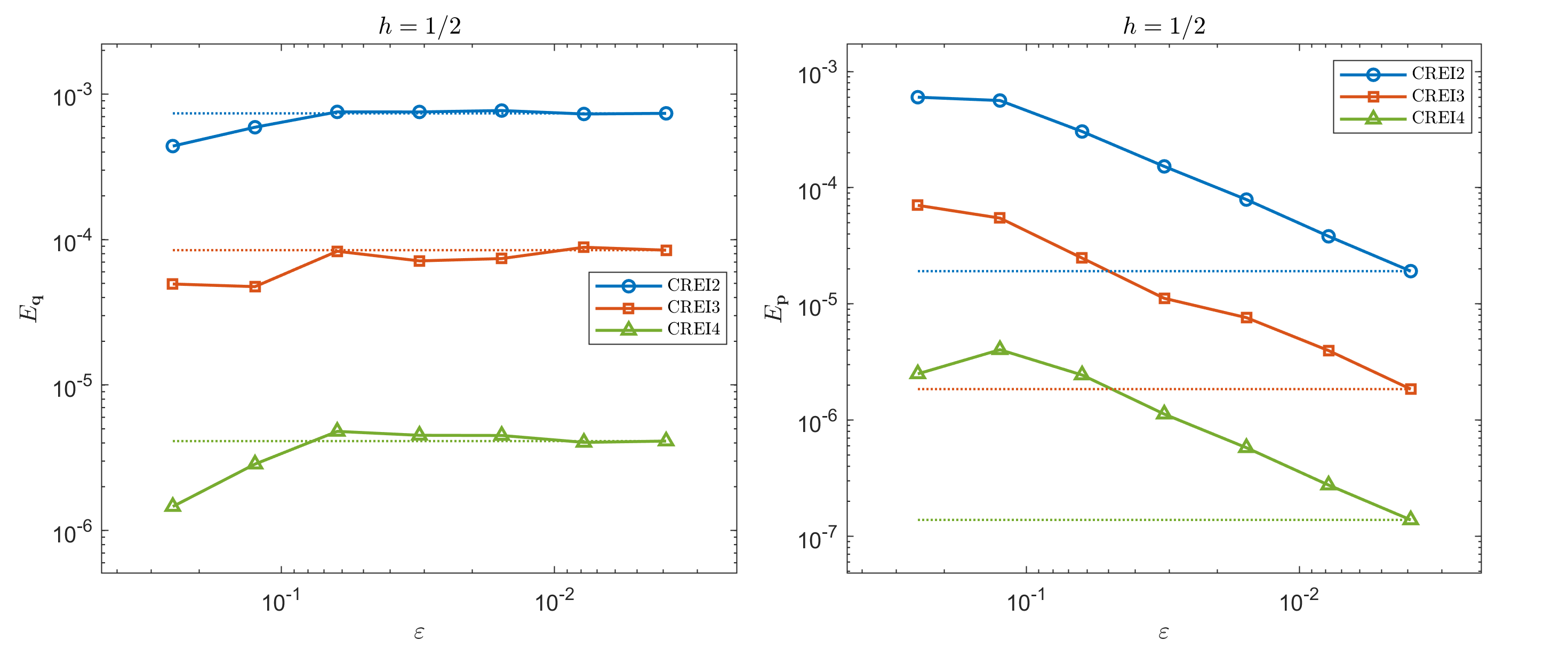}
\caption{Errors \(E_{\mathit{\bq}}\) and \(E_{\mathit{\bp}}\) versus \(\eps\) for fixed \(h=1/2\). The dotted lines are horizontal references for the small-\(\eps\) error levels of CREI2, CREI3, and CREI4.}
\label{fig:fixed-h-ep}
\end{figure}

\FloatBarrier

Finally, we compare CREI2 with the classical fourth-order Runge--Kutta method applied directly to the original \((\bx,\bv)\) system. We use \(h=2^{-7}\) and \(h=2^{-8}\) and vary \(\eps\). Figure \ref{fig:rk4-comp} shows that CREI2 remains substantially more uniform in \(\eps\), while RK4 deteriorates when the step does not resolve the fast rotation.

\begin{figure}[htbp]
\centering
\includegraphics[width=0.95\textwidth]{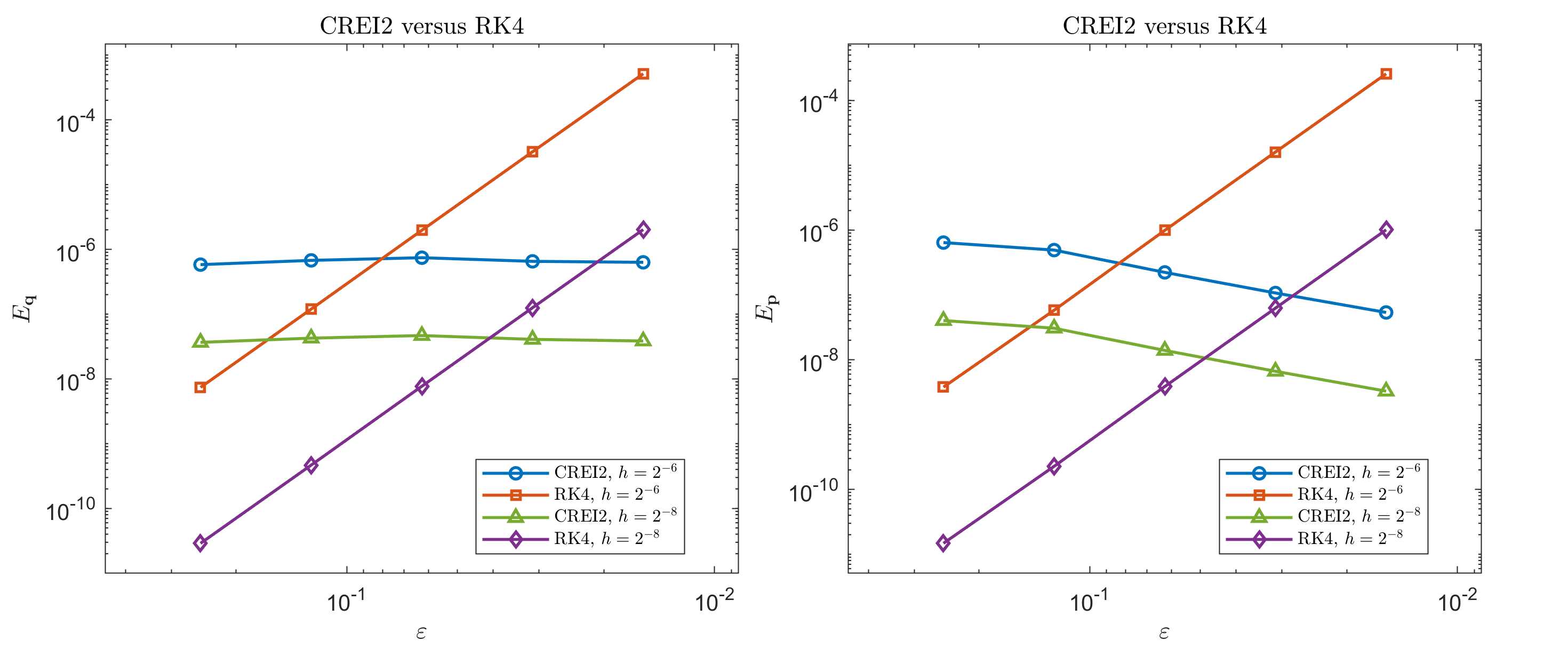}
\caption{Comparison between CREI2 and the classical RK4 method for two fixed step sizes. The left and right panels show \(E_{\mathit{\bq}}\) and \(E_{\mathit{\bp}}\), respectively.}
\label{fig:rk4-comp}
\end{figure}

\FloatBarrier

\section{Conclusion}\label{sec:conclusion}

For the three-dimensional nonrelativistic charged-particle model, the kernel-range transformation produces a semisimple zero block and a constant skew-symmetric oscillatory block. CREI removes the homogeneous rotation before Taylor expansion and applies the monomial extension to the remaining deviation. The resulting method has global error \(O(h^k)\) in the two parameter regimes considered above. The numerical experiments confirm the predicted order and the uniform behavior as \(\eps\) decreases.

\clearpage
\bibliographystyle{plain}
\bibliography{ref}

\end{document}